\RequirePackage[l2tabu,orthodox]{nag}
\documentclass{amsart}

      \title[Modules over categories and   Betti posets of monomial ideals]%
            {Modules over categories and\\ Betti posets of monomial ideals}
     \author{Alexandre Tchernev}
     \author{Marco Varisco}
    \address{Department of Mathematics and Statistics, 
             University at Albany, SUNY, USA\smallskip}
      \email{\href{mailto:atchernev@albany.edu}
                  {atchernev@albany.edu}}
    \urladdr{\url{http://www.albany.edu/~tchernev/}\smallskip}
      \email{\href{mailto:mvarisco@albany.edu}
                  {mvarisco@albany.edu}}
    \urladdr{\url{http://www.albany.edu/~mv312143/}\smallskip}
   \keywords{}
  \subjclass[2010]{Primary: 13D02, 05E40, 06A11.\\\indent To appear in \emph{Proceedings of the American Mathematical Society}\smallskip}
       \date{September 30, 2014}

\usepackage{amssymb, mathtools, enumitem, verbatim,
            mathrsfs, tikz, tikz-cd,
            todonotes, etoolbox, hyperref}
\usepackage[alphabetic, backrefs, msc-links, nobysame]{amsrefs}

\hypersetup{pdfinfo={
      Title={Modules over categories and Betti posets of monomial ideals},
     Author={Alexandre Tchernev and Marco Varisco},
   Keywords={},
    Subject={2010 MSC: Primary 13D02, 05E40, 06A11},
}}

\setlist{labelindent=\parindent, leftmargin=*}

\DeclareMathAlphabet{\matheurm}      {U}{eur}{m}{n}
    \SetMathAlphabet{\matheurm}{bold}{U}{eur}{b}{n}

\numberwithin{equation}{section}

\theoremstyle{plain}
  
  \newtheorem{corollary}  [equation]{Corollary}
  \newtheorem{lemma}      [equation]{Lemma}
  \newtheorem{proposition}[equation]{Proposition}
  \newtheorem{theorem}    [equation]{Theorem}

\theoremstyle{definition}
  \newtheorem{definition} [equation]{Definition}
  \newtheorem{example}    [equation]{Example}
  \newtheorem{remark}     [equation]{Remark}

\newcommand*{\define}[5]{%
  \ifstrequal{#2}{*}{\expandafter#1\expandafter*}
                    {\expandafter#1}%
  \csname#4#5\endcsname{#3{#5}}
}
\forcsvlist{\define{\newcommand}{*}{\mathbf}{B}}{%
  A,B,C,D,E,F,G,H,I,J,K,L,M,N,O,P,Q,R,S,T,U,V,W,X,Y,Z,%
  a,b,c,d,e,f,g,h,i,j,k,l,m,n,o,p,q,r,s,t,u,v,w,x,y,z%
}
\forcsvlist{\define{\newcommand}{*}{\boldsymbol}{BS}}{%
  a,b,c,d,e,f,g,h,i,j,k,l,m,n,o,p,q,r,s,t,u,v,w,x,y,z%
}
\forcsvlist{\define{\newcommand}{*}{\mathcal}{C}}{%
  A,B,C,D,E,F,G,H,I,J,K,L,M,N,O,P,Q,R,S,T,U,V,W,X,Y,Z,%
}
\forcsvlist{\define{\newcommand}{*}{\mathscr}{S}}{%
  A,B,C,D,E,F,G,H,I,J,K,L,M,N,O,P,Q,R,  T,U,V,W,X,Y,Z,%
}                                  
\renewcommand*{\SS}{\mathscr{S}}
\forcsvlist{\define{\newcommand}{*}{\matheurm}{E}}{%
  A,B,C,D,E,F,G,H,I,J,K,L,M,N,O,P,Q,R,S,T,U,V,W,X,Y,Z,%
  a,b,c,d,e,f,g,h,i,j,k,l,m,n,o,p,q,r,s,t,u,v,w,x,y,z%
}
\forcsvlist{\define{\newcommand}{*}{\mathfrak}{F}}{%
  A,B,C,D,E,F,G,H,I,J,K,L,M,N,O,P,Q,R,S,T,U,V,W,X,Y,Z,%
  a,b,c,d,e,f,g,h,i,j,k,l,m,n,o,p,q,r,s,t,u,v,w,x,y,z%
}
\forcsvlist{\define{\newcommand}{*}{\mathbb}{I}}{%
  A,B,C,D,E,F,G,H,I,  K,L,M,N,O,P,Q,R,S,T,U,V,W,X,Y,Z%
}                

\forcsvlist{\define{\newcommand}{*}{\matheurm}{}}%
           {ch, fun, mg, sets, spaces}
\newcommand*{\modu}[1]{#1\D\matheurm{mod}}

\forcsvlist{\define{\DeclareMathOperator}{}{}{}}%
           {aut, chhom, coker, ev, id, map, mor, obj, 
            pt, rk, sk, Tor}

\forcsvlist{\define{\DeclareMathOperator}{*}{}{}}%
           {colim}

\newcommand*{\op}  {{\operatorname{op}}}

\newcommand*{\gr}  {{\operatorname{gr}}}

\newcommand*{\TO}[1][]{\stackrel{#1}{\to}}
\newcommand{\MOR}[4][]{#2\colon#3\TO[#1]#4}

\DeclarePairedDelimiterX\SET[2]{\{}{\}}
                               {\,#1\;\delimsize\vert\;#2\,}

\DeclareMathOperator*{\disju} {\ts\coprod}

\DeclareMathOperator*{\tensor}{\otimes}

\newcommand*{\lra}{\longrightarrow} 
\newcommand*{\lla}{\longleftarrow}

\newcommand*{\ts}{\textstyle}

\newcommand*{\os}{\overset}
\newcommand*{\ns}{\negthickspace}
\newcommand*{\ths}{\thinspace}

\newcommand*{\bfa}{\alpha}
\newcommand*{\bfb}{\beta}
\newcommand*{\bfg}{\gamma}

\newcommand*{\D}{\text{-}}

\newcommand*{\gf}{\Bbbk}
\newcommand*{\m}{\mathfrak{m}}

\newcommand*{\poset}[1]{\mathcal{#1}}
\newcommand*{\KP}[1]{{#1}^{\poset{P}}}
\newcommand*{\KB}[1]{{#1}^{\poset{B}}}

\newcommand*{\cx}[2][]{#2_{\bullet #1}}
\newcommand*{\cxB}[2][]{#2_{\bullet #1}^{\ths\poset{B}}}
\newcommand*{\cxP}[2][]{#2_{\bullet #1}^{\ths\poset{P}}}

\DeclareMathOperator{\HH}{H} 
\DeclareMathOperator{\RH}{\widetilde{H}}

\begin{document}

\begin{abstract} 
We introduce to the context of multigraded modules the methods 
of modules over categories from algebraic topology 
and homotopy theory. We develop the basic theory quite generally, 
with a view toward future applications to a wide class of graded 
modules over graded rings in~\cite{Tchernev-Varisco}. 
The main application in this paper is to study the Betti poset 
$\poset{B}=\poset{B}(I,\gf)$ of a monomial ideal~$I$ in the 
polynomial ring $R=\gf[x_1,\dots,x_m]$ over a field~$\gf$, which 
consists of all degrees in~$\IZ^m$ of the homogeneous basis 
elements of the free modules in the minimal free $\IZ^m$-graded 
resolution of~$I$ over~$R$. We show that the order simplicial 
complex of~$\poset{B}$ supports a free resolution of $I$ over $R$.
We give a formula for the Betti numbers of $I$ in terms of Betti 
numbers of open intervals of $\poset{B}$, and we show that 
the isomorphism class of $\poset{B}$ completely determines the 
structure of the minimal free resolution of $I$, thus  
generalizing with new proofs results of Gasharov, Peeva, and 
Welker~\cite{Gasharov-Peeva-Welker}. We also characterize the 
finite posets that are Betti posets of a monomial ideal. 
\end{abstract}

\maketitle

\section*{Introduction}

Monomial ideals have long been the focus of extensive research
as fundamental objects 
that provide a gateway for interaction between commutative 
algebra, combinatorics, symbolic computation, algebraic 
geometry, and algebraic topology. The more general case  
of multigraded modules has 
recently become of particular interest in applied algebraic 
topology, where they have emerged as a 
central object of study in the theory of multidimensional 
persistence \cite{Carlsson-Zomorodian}. 
In that theory, as well as in other open questions 
like the Stanley depth conjecture, see e.g. \cites{Apel, Apel-2}, 
one would like to extract 
useful numerical and homological invariants out of data on the 
behaviour of the homogeneous components of the multigraded module 
under study. For monomial ideals these kinds of problems are 
usually handled by exploiting the close connection with 
combinatorics of simplicial complexes and then using the sophisticated 
techniques of topological combinatorics that become available. 
While that approach has been shown to work very well for multigraded 
modules in generic situations \cite{Charalambous-Tchernev}, 
it is quite apparent 
that to effectively tackle these problems 
in general one needs new methods that rely less on combinatorics 
and more on homological algebra.   

The main goal of this paper is to develop such a method 
by  adapting to the setting of 
graded modules over a polynomial ring $R=\gf[x_1,\dots, x_m]$
the language and ideas of modules over categories, 
which are familiar notions 
in algebraic topology and homotopy theory.
When $\CJ$ is a small category a \emph{$\gf$-module over $\CJ$} or a 
\emph{$\gf\CJ$-module} is just a functor from $\CJ$ to the category 
$\modu{\gf}$ of all $\gf$-modules. 
The category $\modu{\gf\CJ}$ has objects the $\gf\CJ$-modules and 
morphisms the corresponding natural transformations. The reason  
such notions are relevant in the study of multigraded modules is 
because one can view  a $\IZ^m$-graded $R$-module $M$    
as a functor from the small category 
of the poset $\IZ^m$ (with coordinate-wise partial ordering) 
to the category $\modu{\gf}$, i.e., a $\gf\IZ^m$-module.  
In particular, the category 
$\mg\modu{R}$ of 
multigraded $R$-modules is equivalent to the category   
$\modu{\gf\IZ^m}$. 
Such categories of functors from small categories to 
categories of modules have been extensively studied in 
the context of algebraic topology, where 
they have proven to be invaluable tools 
in understanding equivariant phenomena. 
There one typically considers functors defined on the 
so-called orbit category of the acting group, see for 
example~\citelist{\cite{tomDieck}*{Section~I.11}, 
\cite{Lueck}*{Section~9}, \cite{Davis-Lueck}, \cite{LRV}}.
For us, the salient feature of this approach comes from 
the fact that $\modu{\gf\CJ}$ 
is an abelian category, 
and therefore we can do homological algebra there. 

In order to use this technique effectively to study particular 
properties of $M$ one needs to identify an appropriate $\CJ$. 
For the applications 
in this paper we need only $\CJ$ 
to be the category of a poset. However, keeping in mind future  
applications to a wide range of graded modules and graded duality 
\cite{Tchernev-Varisco}, 
we also develop the basic theory for modules over  
general small categories $\CJ$. A main contribution 
that is specific to the case 
when $\CJ$ is the category of a poset $\poset{P}$ 
is the introduction and study of a pair of adjoint functors: 
the exact \emph{$\poset P$-sampling} functor 
\[
\MOR{\KP{(-)}}{\mg\modu{R}}{\modu{\gf\poset P}}, 
\]
and its left adjoint the \emph{$\poset P$-homogenization} functor 
\[
\MOR{\SR_\gr\tensor_{\gf\poset P}(-)}
{\modu{\gf\poset P}}{\mg\modu{R}}
.
\]
These allow, among other things, to give functorial 
descriptions of common homogenization and relabeling techniques 
currently in the literature.   
We  emphasize that 
both functors have a simple explicit definition and are  
straightforward to compute. 

The main application is to    
study the structure of the \emph{Betti poset} $\poset{B}$ 
of a monomial ideal $I$ in $R$ (over a fixed field $\gf$), 
which consists of the $\IZ^m$-degrees of the 
basis elements 
of the free modules in the minimal free $\IZ^m$-graded 
resolution of $I$ over $R$. 
In our first main result, Theorem~\ref{T:main-1}, we show
that the order 
simiplicial complex $\Delta(\poset{B})$ of the Betti poset 
supports a free resolution of $I$ over $R$. 
Next, we show in 
Theorem~\ref{T:main-2} and Theorem~\ref{T:main-3} that 
the lcm-lattice in the results of Gasharov, Peeva, and 
Welker in~\cite{Gasharov-Peeva-Welker}
can be replaced by the 
Betti poset of the monomial ideal. In particular, 
the isomorphism class of the 
Betti poset completely determines the structure of the minimal free 
resolution of the monomial ideal. 
We include a 
simple example of two ideals that do not have isomorphic 
lcm-lattices, but have isomorphic Betti posets 
over every field $\gf$.  
Finally, we characterize in 
Theorem~\ref{T:when_is_betti}
the finite posets that are Betti posets of monomial ideals. 

It should be noted that while one can, as expected, deduce our 
Betti poset results also from the lcm-lattice results in 
\cite{Gasharov-Peeva-Welker} via judicious use of the   
homology version \cite{Bjorner-Wachs-Welker} 
of Quillen's Fiber Lemma and its consequences, the point of 
our paper is that 
the theory we develop and the proofs we give rely only on 
basic category theory and basic homological algebra. 
This allows us to use the same approach in 
\cite{Tchernev-Varisco} to study other cases of 
classical interest such as toric ideals, where there are 
currently no notions analogous to that of the lcm-lattice.

We would like to thank Amanda Beecher and Timothy Clark for 
useful conversations related to the material in 
Section~\ref{S:when-is-betti}. 


\section{Preliminaries}

Throughout this paper $\gf$ is a commutative,
associative, unital ring, modules are unitary, 
unadorned tensor products are over~$\gf$, and  
unless otherwise specified all functors are understood 
to be covariant.
If $W$ is any set, we write $\gf[W]$ for the free
$\gf$-module with basis the set~$W$. We denote by
$\modu{\gf}$ the category of $\gf$-modules. Given
an additive category~$\CA$ (like for
example~$\modu{\gf}$) we denote by~$\ch(\CA)$ the
additive category of chain complexes in~$\CA$ and
chain homomorphisms. For us chain complexes are
always understood to be indexed over the
integers~$\IZ$, and the differentials in a
complex~$X$ decrease degree. Given any poset~$\poset{P}$
we view it as a small category whose objects are
exactly the elements of~$\poset{P}$ and in which there
is exactly one morphism from $p$ to~$q$ if and
only if $p\leq q$, and none otherwise. When $p\le
q$ in $\poset P$ we abuse notation and denote the unique
morphism from $p$ to $q$ also by $p\le q$. 
Notice that
order-preserving functions between posets,
i.e., morphisms of posets,
correspond exactly to functors between the
associated categories. We always use
the same notation for a poset and its corresponding 
category. For each $a\in\poset{P}$ we write 
$\poset{P}_{\le a}$ for the filter
$\{x\in\poset{P}\mid x\le a\}$ in~$\poset{P}$. The filter
$\poset{P}_{<a}$ is defined analogously.
We denote by~$\IN$ the set
of all non-negative integers, and for every
$n\in\IN$ we write~$[n]$ for the totally ordered
set~$\{\,0<1<\dotsb<n\,\}$.  
We consider $\IN^m$ and $\IZ^m$ as
posets via the coordinatewise partial order:
$(a_1,\dots, a_m)\le (b_1,\dots, b_m)$ if and only
if $a_i\le b_i$ for all~$i$. In particular, both posets are  
lattices with joins given by taking componentwise 
maximums. 
A \emph{$\IZ^m$-graded poset} is a 
poset $\poset P$
together with a morphism of posets 
$\MOR{\gr}{\poset P}{\IZ^m}$ called the \emph{grading morphism}.  
When $\poset{P}$ is a subposet 
of~$\IZ^m$ then we always consider it as $\IZ^m$-graded with 
grading morphism the inclusion map.

Now let 
$R=\gf[x_1,\dotsc,x_m]$ be a polynomial ring over $\gf$ 
in the $m$ variables $x_1,\dotsc,x_m$. 
We consider the  
ring $R$ with the canonical $\mathbb Z^m$-grading,
called \emph{multigrading}.
For each 
$\bfa=(a_1,\dotsc,a_m)\in\mathbb N^m$ 
we write $x^{\bfa}$ for the monomial
$x_1^{a_1}\dotsm x_m^{a_m}$. 
Thus we have $R=\bigoplus_{\bfa\in\IZ^m}R_{\bfa}$, where 
\[
R_{\bfa}=
\begin{cases}
\gf x^{\bfa} &\text{ if } \bfa\in\IN^m; \\ 
0            &\text{ otherwise.} 
\end{cases}
\]
Let $M=\bigoplus_{\bfa\in\IZ^m}M_{\bfa}$ be a multigraded 
$R$-module.
We denote by~$\mg\modu{R}$ the category of multigraded 
$R$-modules (also called monomial graded $R$-modules) 
and homogeneous $R$-linear homomorphisms of degree~$0$.
We set 
\[
\deg(M)=\{\bfa\in\IZ^m\mid M_{\bfa}\ne 0\}. 
\]
If $\bfg\in\IZ^m$ we write  $M(\bfg)$ for 
the corresponding degree-shifted $R$-module, i.e., 
$M(\bfg)_{\bfa}=M_{\bfa+\bfg}$. In particular, $R(-\bfg)$ 
stands for the free multigraded $R$-module of rank one  
generated by a single free generator 
$e=1_R\in R(-\bfg)_{\bfg}$ 
of (multi)degree~$\bfg$.  A free multigraded $R$-module is  
then just a direct sum 
$\bigoplus_{\bfa\in\IZ^m}R(-\bfa)^{b_{\bfa}}$.    
Let
\[
\cx{F}=\qquad
0 \lla F_0 \lla F_1 \lla \dotsb \lla F_d \lla \dotsb 
\] 
be a free multigraded chain complex  
over $R$, i.e.,
the free modules $F_k$ are  
free multigraded and the differentials of $\cx{F}$ 
are morphisms of multigraded 
modules. For each $k$ let $B_k$ be a homogeneous 
basis of $F_k$, and for $\bfa\in\mathbb Z^m$ 
write $B_{k,\bfa}$ for the set of basis elements in $B_k$ 
of multidegree $\bfa$. Write $\cx[\bfa]{F}$ 
for the multigraded strand 
\[
\cx[\bfa]{F} =\qquad
0 \lla (F_0)_{\bfa} \lla (F_1)_{\bfa} \lla \dotsb \lla 
                        (F_d)_{\bfa} \lla \dotsb
\]
of $\cx{F}$ in degree $\bfa$. It is straightforward 
to notice that each $\gf$-module $(F_d)_{\bfa}$ is 
free over $\gf$ with basis the 
set $\coprod_{\bfg\le\bfa}\{x^{\bfa-\bfg}b\mid b\in B_{d,\bfg}\}$. 
Since the differentials in $\cx{F}$ 
preserve multidegrees 
the chain complex  $\cx{F}$ 
decomposes into a direct sum of strands 
$
\cx{F}=
\bigoplus_{\bfa\in\mathbb Z^m}\cx[\bfa]{F},  
$
and is a free resolution of $M$ if and only if each strand 
$\cx[\bfa]{F}$ is a free resolution of $M_\bfa$ over $\gf$. 
Finally, 
for any $\bfa\in\IZ^m$ and $\bfb\in\IN^m$ we have a 
canonical injective 
morphism of chain complexes 
\[
x^{\bfb}\colon \cx[\bfa]{F}\lra \cx[\bfa+\bfb]{F}
\]
via multiplication by the monomial $x^{\bfb}$. 

A principal case of interest is when $\gf$ is a field, 
and 
$\cx{F}$ is a minimal free resolution 
of $M$ over $R$. In that case  
the integer $\beta_{d,\bfa}=\beta_{d,\bfa}(M)=|B_{d,\bfa}|$ 
is called the $d$th 
\emph{Betti number of $M$ in multidegree $\bfa$} and 
\[
\beta_{d,\bfa}=
\dim_\gf\bigl(\cx[\bfa]{F}/
             \cx[\bfa]{F}\cap\m\cx{F}\bigr)_d =
\dim_\gf\Tor^R_d(M, R/\m)_{\bfa}  
\]
where as usual 
$\Fm=(x_1,\dots,x_m)$ is the maximal 
ideal generated by the variables in the polynomial 
ring $R$. 
A main motivation for developing the methods in this 
paper and in \cite{Tchernev-Varisco} was the desire to 
understand how the properties of the 
following poset relate to the properties of $M$.

\begin{definition}
Let $\gf$ be a field and let $M$ be a multigraded  
$R$-module. The set 
\[
\poset{B}(M)=
\{\bfa\in\IZ^m \mid \beta_{d,\bfa}(M)\ne 0 
                    \text{ for some } d\}  
\]
is called the set of \emph{Betti degrees} of $M$.
We consider it as a poset,
and call it then the \emph{Betti poset} of $M$,  
with respect to the partial ordering 
induced by the 
partial ordering on~$\IZ^m$. 
\end{definition}

The applications we consider here  are 
when $M$ is a monomial ideal~$I$ in~$R$.
In this case $\poset{B}(I)$ as defined above is the same as the Betti poset of~$I$ defined in~\cite{Clark-Mapes1} minus its minimal element.
When we want to emphasize the role of the field $\gf$ we write $\poset{B}(I,\gf)$ for $\poset{B}(I)$.
 
An important combinatorial object associated with $I$ is 
the \emph{lcm-lattice} $\poset L=\poset L(I)$, which is the 
subposet of $\IZ^m$ join-generated in $\IN^m$ by the multidegrees 
of the minimal generators of $I$. 
It is well known that the Betti poset $\poset B(I,\gf)$ 
is a subposet of  $\poset L(I)\setminus\hat 0$ (we 
write $\hat 0$ for the minimal element of a lattice). 
As a consequence of our definition the Betti 
poset of $I$ does not in general contain a smallest element. Its 
minimal elements are exactly the multidegrees of the 
minimal generators of $I$.  

Since the results in \cite{Clark1} about the first author's 
poset construction,  
the significance of posets other than the lcm-lattice 
for the study of free resolutions of monomial ideals has 
become more apparent, see \cites{Clark2, Clark-Tchernev}.  
Representing a natural next step in this line of research,  
the Betti poset of a monomial ideal was introduced  
in \cite{Clark-Mapes1}. The Betti poset results we present 
here were quickly followed by those 
in \cite{Clark-Mapes2}, and more recently the Betti poset 
was investigated in \cite{Wood}. 
We apply our newly developed tools to show
that, for the purposes of describing the minimal free 
resolution of $I$ over a fixed field $\gf$, the 
corresponding Betti poset encodes all the required 
information.


\section{Modules over a category}
\label{S:modules_over_category} 

In this section we introduce the fundamental concept of 
modules over a category, the sampling process, and we 
describe several examples.
We denote by~$\CJ$ a small category, e.g., the category 
associated with a poset.

\begin{definition}
\label{def:kJ-mod}
A \emph{$\gf\CJ$-module} is a functor 
$\MOR{M}{\CJ}{\modu{\gf}}$.
A \emph{homomorphism} of $\gf\CJ$-modules is a natural 
transformation.
So the category of $\gf\CJ$-modules, denoted 
$\modu{\gf\CJ}$, is just 
the category of functors from $\CJ$ to~$\modu{\gf}$; 
in symbols, $\modu{\gf\CJ}=\fun(\CJ,\modu{\gf})$.
\end{definition}

\begin{example}\label{ex:trivial-examples}
Here are some trivial examples.
If $\CJ=[0]$ is the category with exactly one object and one 
(identity) morphism, then obviously 
$\modu{\gf[0]}=\modu{\gf}$.
If $G$ is a group and $\underline{G}$ is the category with 
only one object and one (invertible) morphism for every 
element of~$G$, with composition defined by multiplication 
in~$G$, then $\modu{\gf\underline{G}}$ is the category 
of left modules over the group ring $\gf[G]$  
and $\modu{\gf\underline{G}^\op}$ is the category of 
right~$\gf[G]$-modules.
This explains the notation and terminology.
\end{example}

\begin{definition}
Let $\poset P$ be a $\IZ^m$-graded poset with grading 
$\gr\colon\poset{P}\lra \IZ^m$.  
Let $M$ be a multigraded $R$-module. 

(a) The \emph{$\poset P$-sample of $M$} is  the 
$\gf\poset P$-module $\KP{M}$ given by 
\[
\KP{M}(a)=M_{\gr(a)} \quad\text{ and }\quad 
\KP{M}(a\le b)=x^{\gr(b)-\gr(a)}\colon M_{\gr(a)}\lra M_{\gr(b)}. 
\]

(b) We refer to a functor of the form $\KP{(-)}$ 
as a \emph{sampling} (functor). Clearly a sampling  
is an exact functor
\[
\MOR{\KP{(-)}}{\mg\modu{R}}{\modu{\gf\poset P}}
\]
from the abelian category of multigraded 
$R$-modules to~$\modu{\gf\poset P}$. 
\end{definition}


\begin{remark} 
When $I$ is a monomial ideal in $R$  
and $\poset{P}$ a subposet of $\IZ^m$ such that 
$
\poset{P}\subseteq\deg(I),
$ 
then the $\gf\poset{P}$-module $\KP{I}$ 
is in fact isomorphic to the constant 
$\gf\poset{P}$-module 
with value $\gf$. Indeed, the isomorphism 
$\iota$ is given objectwise for each $\bfa$ 
by the isomorphisms $\iota(\bfa)\colon \gf\lra I_{\bfa}$ 
via the obvious formula $c\longmapsto cx^{\bfa}$.    
\end{remark}

Notice that $\modu{\gf\CJ}$ is an abelian 
category. 
Kernels and images are computed objectwise.
A sequence of $\gf\CJ$-modules~$L\TO M\to N$ 
is exact if and only if $L(j)\TO M(j)\TO N(j)$ 
is exact for every $j\in\obj\CJ$.
In particular it makes sense to speak of 
projective $\gf\CJ$-modules, for example, 
and to consider chain complexes of $\gf\CJ$-modules.
Notice that a chain complex of $\gf\CJ$-modules 
can equivalently be thought of as functor 
from~$\CJ$ to~$\ch(\modu{\gf})$; in symbols, 
$
\ch(\modu{\gf\CJ})=\fun\bigl(\CJ,\ch(\modu{\gf})\bigr).
$

\begin{example}\label{E:mfr-functor}
Let $\gf$ be a field, and 
let $\poset{B}$ be the Betti poset of 
a multigraded $R$-module $M$. 
The minimal free resolution 
$\cx{F}$ of $M$ over $R$ yields a 
$\gf\poset{B}$-chain complex 
$\cxB{F}$ given by   
\[
\cxB{F}(\bfa)= \cx[\bfa]{F} \quad\text{and}\quad  
\cxB{F}(\bfa\leq\bfb) = 
x^{\bfb-\bfa}\colon \cx[\bfa]{F}\lra \cx[\bfb]{F} 
\]
and for each $n\ge 0$ a $\gf\poset{B}$-module $\KB{F_n\ns}$ 
given by  
\[
\KB{F_n\ns}(\bfa)= (F_n)_{\bfa} \quad\text{and}\quad  
\KB{F_n\ns}(\bfa\leq\bfb) = 
x^{\bfb-\bfa}\colon (F_n)_{\bfa}\lra
             (F_n)_{\bfb}.
\]
\end{example}

Next we show how standard 
topological and combinatorial constructions used 
in the literature to study free resolutions of 
multigraded $R$-modules can be interpreted as 
$\gf\poset{P}$-chain complexes. 

\begin{example}\label{T:simplicial}
(Simplicial chain complexes) \  
Let $\Delta$ be a simplicial complex, and let 
$\poset{P}=\poset{P}(\Delta)$ be its face poset. 
For each face $H$ let $\Delta_H$ be the 
subcomplex formed by taking all faces contained in $H$.
Clearly $H_1\subseteq H_2$ if and only if 
$\Delta_{H_1}\subseteq\Delta_{H_2}$. By taking simplicial 
chain complexes with coeficients in $\gf$ we obtain 
a $\gf\poset{P}(\Delta)$-chain complex 
$\cx{\CS}=\cx{\CS}(\Delta,\gf)$ given by 
\[
\cx{\CS}(H)=\cx{C}\bigl(\Delta_H; \gf\bigr), 
\] 
where for $H_1\le H_2$ the morphism 
$
\cx{\CS}(H_1\le H_2)\colon 
\cx{C}\bigl(\Delta_{H_1}; \gf\bigr) \lra 
\cx{C}\bigl(\Delta_{H_2}; \gf\bigr) 
$
is the morphism of simplicial chain complexes 
induced by the inclusion 
$\Delta_{H_1}\subseteq\Delta_{H_2}$. 
\end{example}

\begin{example}
(Frames)  \ 
Let $\cx{U}=(U_k,\partial_k)$ 
be a chain complex of based free $\gf$-modules; in 
particular this includes the case of a \emph{frame} as 
defined in \cite{Peeva-Velasco} and therefore also covers 
the simplicial chain complex case from 
Example~\ref{T:simplicial}, 
and the case of a cellular chain complex of a CW-complex. 
Let $B_k$ be the fixed basis of $U_k$ and let 
$B=\disju_kB_k$. Let 
$\poset{P}$ be any poset structure on $B$ such that if 
$b\in B_k$ and $\partial_k(b)=\sum_{c\in B_{k-1}}a_cc$ with 
$a_c\ne 0$ then $b>c$ in $\poset{P}$. For any 
$b\in\poset{P}$ we write $U_k(b)$ for the free submodule of 
$U_k$ with basis the set 
$
B_k(b)=\{c\in B_k\mid c\le b\}.
$
Then clearly $\cx{U}(b)=(U_k(b),\partial_k)$ is a 
subcomplex of $\cx{U}$. Therefore we obtain a 
$\gf\poset{P}$-chain complex 
$\cx{\CF}=\cx{\CF}(\cx{U})$ given by 
\[
\cx{\CF}(b)=\cx{U}(b), 
\]
where for $c\le b$ in $\poset{P}$ the morphism 
$
\cx{\CF}(c\le b)\colon \cx{U}(c)\lra \cx{U}(b)
$
is just the inclusion $\cx{U}(c)\subseteq\cx{U}(b)$.  
\end{example}


\begin{example}
\label{ex:ordercx}
Let $\poset{P}$ be a poset, and let $\Delta(\poset{P})$ 
be the 
\emph{order simplicial complex} of $\poset{P}$.
The $n$-faces of $\Delta(\poset{P})$
are all strictly increasing chains 
$
A=\{a_0 < \dots < a_n\}
$ 
in $\poset{P}$. We always consider such a face $A$ 
with the orientation given by ordering its vertices in 
increasing order,  and  we  
call the maximal element $a_n=\max A$ 
the \emph{apex} of the face $A$. Now we 
obtain a natural $\gf\poset{P}$-chain complex 
$\cx{\CE}=\cx{\CE}(\poset{P},\gf)$ by taking 
simplicial chain complexes as follows: 
\[
\cx{\CE}(a)=\cx{C}\bigl(\Delta(\poset{P}_{\le a}); \gf\bigr)
\] 
where $\Delta(\poset{P}_{\le a})$ is the subcomplex 
of $\Delta(\poset{P})$ with faces those chains $A$ such that 
$\max A\le a$. When
$a\le b$ in $\poset{P}$, the morphism 
\[
\cx{\CE}(a\le b)\colon 
\cx{C}\bigl(\Delta(\poset{P}_{\le a}); \gf\bigr) \lra 
\cx{C}\bigl(\Delta(\poset{P}_{\le b}); \gf\bigr) 
\]
is defined to be the morphism of simplicial chain complexes 
induced by the inclusion  
$\poset{P}_{\le a}\subseteq \poset{P}_{\le b}$. 
By taking simplicial $n$-chains, we also obtain for each $n$ 
a $\gf\poset{P}$-module $\CE_n=\CE_n(\poset{P},\gf)$ with  
\[
\CE_n(a)=C_n\bigl(\Delta(\poset{P}_{\le a}); \gf\bigr).
\]
Our main goal for the next
two sections will be to show that this 
produces a canonical projective resolution of the 
constant $\gf\poset{P}$-module with value $\gf$. 
This fact is an important ingredient in the proofs 
of our main results. 
\end{example}


\section{Tensor products}
\label{S:tensor_products}

One should think of the sampling functors $\KP{(-)}$ and 
$\KB{(-)}$ from 
the previous section as sophisticated dehomogenization 
tools. Applying them  
forgets about those homogeneous components of our multigraded 
modules  whose degrees are not coming from 
$\poset P$ and $\poset{B}$, 
respectively. Furthermore, they 
have the effect of stripping from the 
remaining homogeneous components 
their actual multidegrees.  The canonical way to recover 
some or all of that lost information is through 
an appropriate notion of a tensor product.

\begin{definition}
Let $N$ be a $\gf\CJ^\op$-module 
and $M$ be a $\gf\CJ$-module.
Define their \emph{tensor product over~$\gf\CJ$} to be 
the $\gf$-module
\[
N\tensor_{\gf\CJ}M=
\left(\bigoplus_{j\in\obj\CJ}N(j)\tensor M(j)\right)/S
\]
where $S$ is the submodule generated by
\[
\SET{nu\tensor m - n\tensor um}
    {m\in M(j), n\in N(i), u\in\mor_\CJ(j,i)}
\]
and, to underscore the analogy with the tensor product 
of modules over rings,  
we use the shorthands $um$ and $nu$ for the values 
of the homomorphisms $M(u)$ and $N(u)$ at the elements 
$m$ and $n$ respectively; thus 
$um=M(u)(m)$ and $nu=N(u)(n)$. 
It is a straightforward consequence of the definition 
that the functor 
$N\tensor_{\gf\CJ}(-)$ preserves epimorphisms and 
direct sums. 
\end{definition}

\begin{example}
Let $G$ be a group, $M$ a left $\gf[G]$-module, and 
$N$ a right $\gf[G]$-module.
As explained in Example~\ref{ex:trivial-examples}, 
we can think of~$N$ as a $\gf\underline{G}^\op$-module 
and of~$M$ as a $\gf\underline{G}$-module.
Then $N\tensor_{\gf\underline{G}}M=N\tensor_{\gf[G]}M$.
\end{example}

Now we explain that for any $\gf\CJ^\op$-module~$N$ the functor 
$N\tensor_{\gf\CJ}(-)$ is right exact.
The proof of this proceeds as in the classical case 
when~$\CJ=[0]$, by adjointness.

\begin{definition}
Let $N$ be a $\gf\CJ^\op$-module 
and $T$ a $\gf$-module.
Define the 
$\gf\CJ$-module~$\hom_\gf(N,T)$ by sending 
$j$~to~$\hom_\gf(N(j),T)$.
Notice that $\hom_\gf(N,T)$ 
is covariant in~$\CJ$ since $N$ is contravariant.
\end{definition}

From the definitions and the usual tensor-hom adjunction for 
modules one sees that the functor
\[
\MOR{\hom_\gf(N,-)}{\modu{\gf}}{\modu{\gf\CJ}}
\]
is right adjoint to 
\[
\MOR{N\tensor_{\gf\CJ}(-)}{\modu{\gf\CJ}}{\modu{\gf}},
\]
i.e., for all $\gf\CJ$-modules~$M$ and all 
$\gf$-modules~$T$ there are natural isomorphisms
\[
\hom_\gf\bigl(N\tensor_{\gf\CJ}M,T\bigr)
\cong
\hom_{\gf\CJ}\bigl(M,\hom_\gf(N,T)\bigr)
.
\]
From this one concludes as usual 
(see for example~\cite{Weibel}*{Theorem~2.6.1 on page~51}) 
that $N\tensor_{\gf\CJ}(-)$ is right exact,
and $\hom_\gf(N,-)$ is left exact.

Notice that we can repeat everything done so far 
in this section when $N$ is a multigraded $R\CJ^\op$-module, 
i.e., a functor $\MOR{N}{\CJ^\op}{\mg\modu{R}}$.
If $N$ is multigraded $R\CJ^\op$-module and $M$ is a 
$\gf\CJ$-module, then their tensor product 
$N\tensor_{\gf\CJ}M$ is a multigraded $R$-module.
Moreover, if $T$ is a multigraded $R$-module, 
then $\hom_R(N,T)$ is a $\gf\CJ$-module, 
where $\hom_R$ denotes homogeneous $R$-linear homomorphisms 
of degree~$0$. Then the functor
\[
\MOR{\hom_R(N,-)}{\mg\modu{R}}{\modu{\gf\CJ}}
\]
is right adjoint to 
\[
\MOR{N\tensor_{\gf\CJ}(-)}{\modu{\gf\CJ}}{\mg\modu{R}}. 
\]
The fundamental example for us of a multigraded 
$R\CJ^\op$-module is defined next.

\begin{definition}
\label{def:shift}
Let $\poset P$ be a $\IZ^m$-graded poset with grading 
$\gr\colon\poset P\lra\IZ^m$. 
We define the $\gf\poset P^{op}$-module $\SR_{\gr}$ via 
the degree-shift operations by setting  
\[
\SR_{\gr}(a)=R(-\gr(a)) \quad\text{and}\quad 
\SR_{\gr}(a\le b)= x^{\gr(a)-\gr(b)}\colon R(-\gr(a))\lra R(-\gr(b))
\] 
for any morphism $a\le b$ in $\poset P^{op}$. We call $\SR_{\gr}$  
the \emph{multidegree shift functor}. 
Notice that the values 
of $\SR_{\gr}$ are not just $\gf$-modules but also 
free multigraded $R$-modules, 
and $\SR_{\gr}$ sends the morphisms of $\poset P^{op}$ 
to morphisms of multigraded $R$-modules;
i.e., $\SR_\gr$ is a multigraded $R\poset P^\op$-module.
When $\poset{P}$ is a subposet of $\IZ^m$ and $\gr$ is the 
inclusion map, we omit the subscript $\gr$ from the 
notation, and write just $\SR$. 
\end{definition}

When $\poset{P}$ is a $\IZ^m$-graded poset the 
sampling process introduced in the previous section 
is a special case of the above described $\hom$-construction.
Indeed, for every multigraded $R$-module~$T$, we have that 
\[
\hom_R(\SR_\gr,T)=\KP{T}
.
\]
Hence the sampling functor
\[
\MOR{\KP{(-)}}{\mg\modu{R}}{\modu{\gf\poset P}}
\]
is right adjoint to 
\[
\MOR{\SR_\gr\tensor_{\gf\poset P}(-)}
{\modu{\gf\poset P}}{\mg\modu{R}}
.
\]
This latter construction is  the main tool that we will 
use to recover information lost during sampling. 

\begin{definition}
Let $\poset P$ be a $\IZ^m$-graded poset. Let $M$ be 
a $\gf\poset P$-module. The multigraded $R$-module 
\[
\SR_{gr}\tensor_{\gf\poset P} M 
\]
is called the \emph{$\poset P$-homogenization} of $M$.
\end{definition}

The process of $\poset P$-homogenization is a functorial 
generalization of the standard homogenization technique 
used to convert a chain complex of free $\gf$-modules 
into a complex of free multigraded $R$-modules. We briefly 
demonstrate how this works in the case of a 
labeling of a simplicial complex $\Delta$ as in  
\cite{Bayer-Peeva-Sturmfels}*{Construction~2.1}.

\begin{example}\label{T:Taylor}
Let $I$ be a monomial ideal in $R$ with a set of minimal 
generators $\{x^{\bfa_0},\dots, x^{\bfa_n}\}$. Let $\Delta$ 
be a simplicial complex on the set $[n]$, let $\poset P$ be 
its face poset, and let $\cx\CS$ be the corresponding 
$\gf\poset P$-chain complex from Example~\ref{T:simplicial}. 
Let $\MOR{\gr}{\poset P}{\IZ^m}$ be the grading morphism 
given by the formula 
\[
\gr(\{i_1,\dots,i_k\})=\bfa_{i_1}\vee\dots\vee \bfa_{i_k},
\]
where $\vee$ denotes join in the lattice $\IN^m$. Then it 
is clear that $\gr$ is a morphism of posets and unravelling 
the definitions shows that the 
$\poset P$-homogenization 
\[
\cx T=\SR_{\gr}\tensor_{\gf\poset P}\cx\CS
\]
is exactly the complex produced by 
\cite{Bayer-Peeva-Sturmfels}{Construction~2.1}. 
In particular, $\Delta$ supports a free resolution of $I$ 
exactly when $\cx T$ is a resolution of $I$. 
\end{example}

Composing an appropriate sampling with an appropriate 
homogenization yields a functorial description of all  
``relabeling'' procedures considered in 
\cite{Gasharov-Peeva-Welker} and \cite{Peeva-Velasco}.   
Here is an example of this in the case of the  
relabeling Construction 3.2 from \cite{Gasharov-Peeva-Welker}.  
One can clearly produce in a similar manner 
examples corresponding to the procedures of $f$-degeneration 
and $f$-homogenization considered in 
\cite{Peeva-Velasco}. 

\begin{example}
Let $I,R,\Delta,\poset P,\gr,\cx\CS$, and $\cx T$ be as in
Example~\ref{T:Taylor}, and let $\poset L$ be the lcm-lattice 
of $I$. Note that the image of $\gr$ 
is $\hat{\poset L}=\poset L\setminus\hat 0$. Suppose 
$\MOR{f}{\hat{\poset L}}{\IZ^t}$ is a map of posets. Let 
$\gr'=f\circ\gr$, let $S=\gf[y_1,\dots,y_t]$ be a 
polynomial ring with the standard $\IZ^t$-grading, and 
let $\SS_{\gr'}$ and $\SS_f$ be the corresponding 
multidegree-shift functors. Thus $\SS_f$ is a 
multigraded $S\hat{\poset L}^{op}$-module, and $\SS_{\gr'}$ is a 
multigraded $S\poset P^{op}$-module.   
Now the $\poset P$-homogenization 
$\cx T'=\SS_{\gr'}\tensor_{\gf\poset P}\cx\CS$ is exactly the 
complex of free multigraded $S$-modules obtained 
from the complex $\cx T$ by applying the 
relabeling procedure from 
\cite{Gasharov-Peeva-Welker}*{Construction~3.2}
using the relabeling map $f$, and one can 
check directly from the definitions (or use 
Proposition~\ref{T:adjunction-iso} below) that also 
$\cx T'=\SS_f\tensor_{\gf\hat{\poset L}}\bigl(\cx T^{\hat{\poset L}}\bigr)$.   
\end{example}

The adjunction between homogenization and sampling yields
for each $\gf\CP$-module $M$ and 
each multigraded $R$-module~$T$ natural homomorphisms 
\[
\MOR{\eta}
{M}{\Bigl(\SR_\gr\tensor_{\gf\poset{P}}M\Bigr)^\poset{P}}
\quad\text{and}\quad
\MOR{\epsilon}
{\SR_\gr\tensor_{\gf\poset{P}}\bigl(\KP{T}\bigr)}{T}, 
\]
the unit and counit of the adjunction, respectively. 

\begin{proposition}\label{T:adjunction-iso} 
Let $\poset{P}$ be a subposet of $\IZ^m$, let $M$ be a 
$\gf\poset{P}$-module, and let $F$ be a free multigraded 
$R$-module such that $\poset{P}$ contains the degrees of 
the elements of some (hence every) homogeneous basis 
of $F$. Then:   
\begin{enumerate}
\item 
The unit of adjunction 
$
\MOR{\eta}{M}
{(\SR\tensor_{\gf\poset{P}}M)^\poset{P}}
$
is 
an isomorphism of $\gf\poset P$-modules.

\item 
The counit of adjunction 
$
\MOR{\epsilon}{\SR\tensor_{\gf\poset{P}}(\KP{F})}{F}
$ 
is an isomorphism of multigraded $R$-modules. 
\end{enumerate}
\end{proposition}

\begin{proof} 
Since $X=\SR\tensor_{\gf\poset{P}}M$ is just the 
quotient of the multigraded $R$-module 
$
G=
\bigoplus_{\bfa\in\poset{P}}
R(-\bfa)\tensor M(\bfa)
$ 
by the multigraded submodule $H$ spanned over $R$ 
by all elements of the form 
\[
\bigl(x^{\bfb-\bfa}\otimes m\bigr) - 
\bigl(1\otimes M(\bfa\le\bfb)(m)\bigr)
\]
with $m\in M(\bfa)$ and $\bfa\le\bfb$ in~$\poset{P}$, 
it is straightforward to see that for each $\bfg\in\poset{P}$ 
we get 
\[
\begin{aligned}
X_{\bfg} 
&=
\Bigl(
\bigoplus_{\bfa\le\bfg}\gf x^{\bfg-\bfa}\tensor M(\bfa)
\Bigr)\Big/ \gf\Bigl\langle 
\bigl(x^{\bfg-\bfa}\tensor m\bigr) - 
\bigl(1\tensor M(\bfa\le\bfg)(m)\bigr) \ 
\Big| \ m\in M(\bfa) 
\Bigr\rangle   \\ 
&= \gf\tensor M(\bfg).   
\end{aligned}
\] 
Since in $X$ we have 
$
x^{\bfg-\bfb}(1\otimes m)= 
x^{\bfg-\bfb}\otimes m = 
1 \otimes M(\bfb\le\bfg)(m)
$ 
whenever $\bfb\le\bfg$ 
and $m\in M(\bfb)$, the isomorphism (1) is immediate. 

Since tensoring by $\SR$ and applying $\KP{(-)}$ 
both preserve 
direct sums, and since $\poset P$ contains the degrees 
of the elements of any homogeneous basis of $F$, 
to prove (2) we just need to show that if 
$\bfa\in\poset{P}$ then  
$
\MOR{\epsilon}
{\SR\tensor_{\gf\poset{P}}\bigl(\KP{R(-\bfa)}\bigr)}
{R(-\bfa)}
$
is an isomorphism. However, it is immediate from the 
definition 
that
the source of~$\epsilon$
is the quotient of the multigraded $R$-module 
$
Y=
\bigoplus_{\bfa\le\bfg\in\poset{P}}
R(-\bfg)\tensor_\gf\gf x^{\bfg-\bfa}
$
by the multigraded $R$-submodule $Z$ generated 
over $R$ by all elements of the form 
$(x^{\bfg-\bfb}\otimes x^{\bfb-\bfa}) - (1\otimes x^{\bfg-\bfa})$
with $\bfa\le\bfb\le\bfg$ in $\poset{P}$, hence 
$Y/Z=R(-\bfa)\tensor\gf$. 
\end{proof}


\section{Free and projective \texorpdfstring{$\gf\CJ$}{kJ}-modules}
\label{S:free_and_projective}

In this section, for any small category $\CJ$ we want to 
define free $\gf\CJ$-modules and bases for them, in such a way that 
free $\gf\CJ$-modules are projective.
To this end we first need to define the 
underlying ``object'' of a $\gf\CJ$-module~$M$.  
The most convenient way of doing so is by 
forgetting not just the $\gf$-module structure on 
each~$M(j)$ but also the homomorphisms $M(j)\TO M(i)$ 
that we have for any morphism~$j\TO i$ in~$\CJ$.
So the underlying object of a $\gf\CJ$-module~$M$ 
is just the collection of sets~$M(j)$ indexed by 
the objects of~$\CJ$.
This data can conveniently be encoded into a functor 
from the discrete category~$\obj\CJ$ (i.e., the 
subcategory of~$\CJ$ where the only morphisms 
are the identities) to~$\sets$, i.e., 
an~$(\obj\CJ)$-set.
This defines a forgetful functor 
\[
\MOR{U}{\modu{\gf\CJ}}{(\obj\CJ)\D\sets}\,.
\]
The key observation is that~$U$ has a left adjoint 
\[
\MOR{L}{(\obj\CJ)\D\sets}{\modu{\gf\CJ}}\,,
\]
that is defined by sending an $(\obj\CJ)$-set~$B$ 
to the $\gf\CJ$-module
\[
LB=
\adjustlimits
\bigoplus_{j\in\obj\CJ}
\bigoplus_{\ B(j)\ }\ \gf[\mor_\CJ(j,-)]
\,.
\]
Notice that there is a natural morphism of 
$(\obj\CJ)$-sets 
$
\MOR{\eta}{B}{ULB}
$
(which will be the unit of the adjunction) that 
for every $j\in\obj\CJ$ sends $b\in B(j)$ 
to~$\id_j$ in the corresponding 
summand~$\gf[\mor_\CJ(j,j)]\subseteq LB(j)$ indexed by~$b$.

\begin{example}
Fix an object $j_0\in\obj\CJ$ and consider 
the $(\obj\CJ)$-set~$B$ given by
\[
B(j)=
\emptyset 
\quad\text{if $j\neq j_0$, and }\quad 
B(j_0)=\pt.        
\]
Then $LB=\gf[\mor_\CJ(j_0,-)]$, and $\MOR{\eta}{B}{ULB}$ 
sends $\pt$~to~$\id_{j_0}$.
\end{example}

In order to prove that $L$ is left adjoint to~$U$, 
i.e, that there are natural bijections
\[
\mor_\CJ(B,UM)\cong\hom_{\gf\CJ}(LB,M)
\]
for all $(\obj\CJ)$-sets~$B$ and $\gf\CJ$-modules~$M$, 
we explain in which 
sense $B$ is a basis for~$LB$.

\begin{definition}
Let $M$ be a $\gf\CJ$-module and let $B$ 
be an~$(\obj\CJ)$-set 
together with a morphism of $(\obj\CJ)$-sets 
$\MOR{\mu}{B}{UM}$.
We say that 
\emph{$M$~is free with basis~${B}\TO[\mu]{UM}$} 
if for every $\gf\CJ$-module~$N$ and every morphism 
of~$(\obj\CJ)$-sets $\MOR{g}{B}{UN}$ there is a 
unique homomorphism of $\gf\CJ$-modules $\MOR{G}{M}{N}$ 
such that $(UG)\circ\mu = g$.
\end{definition}

The following lemma is standard from the 
definitions and Yoneda's lemma.

\begin{lemma}
(a) 
If $B$ is an $(\obj\CJ)$-set then the 
$\gf\CJ$-module~$LB$ is free with 
basis $\MOR{\eta}{B}{ULB}$.
Therefore the functor~$L$ is left adjoint 
to the forgetful functor~$U$. 

(b) Free $\gf\CJ$-modules are projectives. \qed
\end{lemma}

\begin{proposition}\label{T:free}
Let $\poset{P}$ be a subposet of $\IZ^m$. 

(a) Let $F$ be a free multigraded $R$-module with homogeneous 
basis $B$ such that $\poset{P}$  
contains the degrees of all the elements of $B$. Then 
the $\gf\poset{P}$-module $\CF=\KP{F}$ 
is free with basis $\KP{B}\os{\mu}\lra U\CF$, 
where the $(\obj\poset{P})$-set $\KP{B}$ is given by 
\[
\KP{B}(\bfa) = B_{\bfa} = 
\{b\in B|\deg(b)=\bfa\} \ \subset 
F_{\bfa}=\CF(\bfa)= U\CF(\bfa), 
\]
and $\mu(\bfa)$ is the inclusion map.  

(b) Let $\CG$ be a free $\gf\poset{P}$-module with basis 
$\CC\os{\mu}\lra U\CG$. Then $\SR\tensor_{\gf\poset{P}}\CG$ 
is a free multigraded $R$-module with homogeneous basis 
$B=\coprod_{\bfa\in\poset{P}}\{1\tensor\mu(c)\mid c\in\CC(\bfa)\}$.  
\end{proposition}

\begin{proof} 
Part (b) follows from part (a). Indeed,  
let $G$ be the free multigraded $R$-module with 
homogeneous basis the set $B$ from part (b). Then 
it is straightforward from the definitions 
in part (a) that 
$\KP{B}\cong\CC$ as $(\obj\poset{P})$-sets, hence 
the free by part~(a) $\gf\poset{P}$-module $\KP{G}$  
is isomorphic to $\CG$. Therefore  
$
G\cong\SR\tensor_{\gf\poset{P}}\KP{G}
 \cong\SR\tensor_{\gf\poset{P}}\CG.
$ 

We proceed with the proof of part (a). 
Let $\CG$ be any $\gf\poset{P}$-module  
and let $g\colon \KP{B} \lra U\CG$ 
be a morphism of $(\obj\poset{P})$-sets. 
We need to show that there is a unique morphism 
$G\colon \CF \lra \CG$ of $\gf\poset{P}$-modules that extends $g$. 
We note that any such extension $G$ of $g$ has to satisfy for 
any $\bfg\leq\bfa$ and any $b\in B_{\bfg}$ the equality 
\begin{equation}
\label{Eq:functoriality} 
G(\bfa)(x^{\bfa-\bfg}b)=
\CG(\bfg\leq\bfa)\bigl(g(\bfg)(b)\bigr). 
\end{equation}  
%
However $\CF(\alpha)= F_{\bfa}$ is 
a free $\gf$-module  
with basis all elements 
$x^{\bfa-\bfg}b$ such that $\bfg\le \bfa$ and $b\in B_{\bfg}$, and therefore  
there exists exactly one homomorphism of $\gf$-modules  
$G(\bfa)\colon \CF(\bfa)\lra \CG(\bfa)$ that satisfies 
\eqref{Eq:functoriality}. 
%
%
\end{proof}

\begin{corollary}
Let $I$ be a monomial ideal with Betti poset 
$\poset{B}$ over a field $\gf$. 
Let $\cx{F}$ be a minimal free multigraded 
resolution of $I$ over $R$. 
Let $\poset{P}$ be a subposet of $\IZ^m$ such that 
$\poset{B}\subseteq \poset{P}$. Then 
the $\gf\poset{P}$-chain complex $\cxP{F}$ is a resolution 
of $\KP{I}$ by free $\gf\poset{P}$-modules. In particular, if 
$\poset{B}\subseteq\poset{P}\subseteq\deg(I)$ then it is a free 
resolution of the constant $\gf\poset{P}$-module with 
value $\gf$.  \qed
\end{corollary}

\begin{proposition}\label{T:resolution-P}
The $\gf\poset{P}$-chain complex $\cx{\CE}(\poset{P},\gf)$ 
defined in Example~\ref{ex:ordercx} is a free (and hence 
projective) resolution of the constant $\gf\poset{P}$-module 
with value~$\gf$.
\end{proposition}

\begin{proof} 
$\cx{\CE}=\cx{\CE}(\poset{P},\gf)$ is a resolution because 
for each $a$ the simplicial complex $\Delta(\poset{P}_{\le a})$ 
is a cone with apex $a$. Also, it is straightforward from 
the definitions that each $\CE_n$ 
is a free $\gf\poset{P}$-module with basis $\MOR{\eta_n}{B_n}{U\CE_n}$ 
where $B_n$ is the $(\obj\poset P)$-set given by 
$
B_n(a)=
\{A \mid A \text{ is an $n$-face of $\Delta(\poset P)$ with } \max A=a\} 
$
and $\eta_n(a)$ is the inclusion map 
$B_n(a)\subset C_n(\Delta(\poset P_{\le a}); \gf) = U\CE_n.$ 
\end{proof}


\section{Main results}
\label{S:main_results}

Throughout this section $\gf$ is a field, 
$R=\gf[x_1,\dots, x_m]$, and $I$ is a monomial 
ideal in $R$ with Betti poset $\poset{B}$ over $\gf$.

\begin{theorem}\label{T:main-1}
Let $\poset{P}$ be any subposet of $\IZ^m$ such that  
$\poset{B}\subseteq\poset P\subseteq \deg(I)$.
Then the order complex $\Delta(\poset P)$ supports a 
free resolution of $I$. 
\end{theorem}

\begin{proof}
Recall from Definition~\ref{def:shift} that 
$
\SR\colon \poset{P}^{\op} \lra \mg\modu{R}
$
is the multidegree shift functor    
given objectwise by $\SR(\bfa)=R(-\bfa)$ and for  
$\bfa\leq\bfb$ in $\poset{P}^{op}$ the corresponding morphism 
$R(-\bfa) \lra R(-\bfb)$ is given by 
multiplication by $x^{\bfa-\bfb}$. 
Let $\cx{\CE}=\cx{\CE}(\poset{P},\gf)$ 
be the free resolution of the constant 
$\gf\poset{P}$-module with value $\gf$ from  
Proposition~\ref{T:resolution-P}. 
Let $\cx{F}$ be the minimal free resolution of $I$ 
over $R$, and let 
$\cx{G}=\SR\tensor_{\gf\poset{P}}\cx{\CE}$. 
Since both $\cxP{F}$ and $\cx{\CE}$ are projective 
resolutions of the constant $\gf\poset{P}$-module 
with value $\gf$, they are chain homotopy 
equivalent and hence so are  
$\cx{F}=\SR\tensor_{\gf\poset{P}}\cxP{F}$ 
and $\cx{G}$. Therefore $\cx{G}$ is a multigraded resolution 
of $I$ over $R$. Finally,   
since in each homological degree $n$ the 
$\gf\poset{P}$-module $\CE_n=\CE_n(\poset{P},\gf)$ is free, 
the corresponding 
$R$-module $G_n=\SR\tensor_{\gf\poset{P}}\CE_n$ is free 
multigraded.  
\end{proof}

\begin{theorem}\label{T:main-2} 
Let $\poset P$ be a subposet of $\IZ^m$ such that 
$\poset B=\poset B(I,\gf)\subseteq\poset P\subseteq\deg(I)$. 

(a) For any $\bfa\in\IZ^m$ and any $d\ge 0$ we have 
$\beta_{d,\bfa}(I)=0$ if $\bfa\notin\poset P$, otherwise 
\[
\beta_{d,\bfa}(I)=
\dim_\gf\RH_{d-1}\bigl(\Delta(\poset{P}_{<\bfa});\gf\bigr).  
\]

(b) In particular, 
$
\poset B(I,\gf)=
\{
a\in\poset P\mid
\RH_k\bigl(\Delta(\poset P_{<a});\gf\bigr)\ne 0
\text{ for some }k\}. 
$

(c) For  any $\bfa\in\poset B(I,\gf)$ the inclusion 
$\Delta(\poset B_{<\bfa})\subseteq\Delta(\poset P_{<\bfa})$ is 
a homology isomorphism over $\gf$. 
\end{theorem}

\begin{proof}
(a) 
Since the multidegrees of the basis elements of 
the free modules in the resolution 
$\cx{G}=\SR\tensor_{\gf\poset{P}}\cx{\CE}(\poset{P}, \gf)$ 
are all in $\poset{P}$, it is immediate 
that $\beta_{d,\bfa}(I)=0$ for $\bfa\notin\poset{P}$. 
If $\bfa\in\poset{P}$ then we have 
$
\Tor_d^R(I,R/\m)_\bfa 
= 
\HH_d(\cx[\bfa]{G}/\m \cx{G}\cap \cx[\bfa]{G}) 
=
\HH_d\Bigl( 
\cx{C}\bigl(\Delta(\poset{P}_{\le\bfa}); \gf\bigr)/ 
\cx{C}\bigl(\Delta(\poset{P}_{<\bfa})  ; \gf\bigr)
\Bigr).     
$
As $\Delta(\poset{P}_{\le\bfa})$ is a 
cone with apex $\bfa$ 
we get  
$
\Tor_d^R(I, R/\m)_\bfa 
\cong 
\RH_{d-1}\bigl(\Delta(\poset{P}_{<\bfa});\gf\bigr). 
$

(b) is immediate from (a) by the definition of Betti poset. 

(c) The inclusion $\poset B\subset\poset P$ induces canonically 
a morphism of $\IZ^m$-graded free resolutions   
$
\cx E = 
\SR\tensor_{\gf\poset{B}}\cx{\CE}(\poset B,\gf) \lra 
\SR\tensor_{\gf\poset{P}}\cx{\CE}(\poset P,\gf) = 
\cx G
$
which is an isomorphism in homology, hence a chain 
homotopy equivalence. Therefore it stays a chain homotopy 
equivalence after tensoring by the $\IZ^m$-graded $R$-module 
$R/\m$, hence induces an isomomorphism 
in homology between the 
corresponding multigraded strands in degree $\bfa$ for each 
$\bfa\in\poset B$. But for these graded strands we have 
$
(\cx E\tensor_R R/\m)_\bfa 
= 
\cx[\bfa]{E}/(\m \cx{E}\cap \cx[\bfa]{E}) 
=
\cx{C}\bigl(\Delta(\poset{B}_{\le\bfa}); \gf\bigr)/ 
\cx{C}\bigl(\Delta(\poset{B}_{<\bfa})  ; \gf\bigr) 
$
and  
$      
(\cx G\tensor_R R/\m)_\bfa 
= 
\cx[\bfa]{G}/(\m \cx{G}\cap \cx[\bfa]{G}) 
=
\cx{C}\bigl(\Delta(\poset{P}_{\le\bfa}); \gf\bigr)/ 
\cx{C}\bigl(\Delta(\poset{P}_{<\bfa})  ; \gf\bigr). 
$
Therefore the map of pairs 
$
\bigl(\Delta(\poset B_{\le\bfa}), \Delta(\poset B_{<\bfa})\bigr)
\lra 
\bigl(\Delta(\poset P_{\le\bfa}), \Delta(\poset P_{<\bfa})\bigr)
$
induced by the inclusion $\poset B\subseteq \poset P$ is an 
isomorphism in relative homology over $\gf$. Since 
$\Delta(\poset B_{\le\bfa})$ and $\Delta(\poset P_{\le\bfa})$ 
are cones with apex $\bfa$, the desired conclusion is 
immediate.  
\end{proof}

Finally, we show that the isomorphism class of 
the Betti poset completely determines 
the structure of the minimal 
free resolution of $I$. The following theorem also 
generalizes, with a new proof,
\cite{Gasharov-Peeva-Welker}*{Theorem~3.3}.  

\begin{theorem}\label{T:main-3}
Let $\poset P$ be a subposet of $\IZ^m$ such that 
$\poset B(I,\gf)\subseteq\poset P\subseteq\deg(I)$. 
Let $\cx{F}$ be a minimal free 
multigraded resolution of $I$ over $R$, and let 
$\cx{\CF}=\cxP{F}$ be the $\poset P$-sample of $\cx F$. 
Let $S=\gf[y_1,\dots, y_t]$ be another polynomial 
ring over the field $\gf$, and let $J$ be a monomial 
ideal of $S$ such that $\poset P$ is isomorphic to 
a subposet $\poset Q$ of $\IZ^t$ with   
$\poset{B}(J,\gf)\subseteq\poset Q\subseteq\deg(J)$. 
Fix one such isomorphism $\gr\colon\poset P\lra\poset Q$. 

(a) The isomorphism $\gr$ maps $\poset B(I,\gf)$ 
isomorphically onto $\poset B(J,\gf)$.  

(b) Viewing $\gr$ as a morphism $\gr\colon\poset P\lra\IZ^t$, 
consider the 
multidegree-shift functor 
$\SS_\gr\colon\poset{P}^{\op}\lra \mg\modu{S}$. 
Then the homogenization 
$\SS_\gr\tensor_{\gf\poset{P}}\cx{\CF}$ is a 
minimal free multigraded 
resolution of $J$ over $S$. 
\end{theorem}

\begin{proof}
Part (a) is immediate from Theorem~\ref{T:main-2}, and 
we proceed with the proof of part (b). 
Identifying $\poset P$ with $\poset Q$,  
we consider $\gr$ as the inclusion map and write $\SS$ for 
$\SS_\gr$.  
Let $\cx{G}$ be a minimal free resolution of $J$ over $S$. 
Then the $\gf\poset{P}$-chain complexes 
$\cx{\CF}$ and $\cxP{G}$ are 
free resolutions of the constant $\gf\poset{P}$-module 
with value $\gf$, hence are chain homotopy equivalent. 
Hence so are the the complexes  
$\cx{\CH}=\SS\otimes_{\gf\poset{P}}\cx{\CF}$ and 
$\SS\tensor_{\gf\poset{P}}\cxP{G}\cong \cx{G}$. 
Therefore $\cx{\CH}$ is a free resolution of $J$ over $S$, 
and since  by Theorem~\ref{T:main-2} and 
Proposition~\ref{T:free} it has 
the correct ranks for its free modules, 
it is a minimal free resolution. 
\end{proof}

\begin{example}
Let $R=\gf[a,b,c,d,e]$ and let $I=(ac,ae,bd,de)$.
The Hasse diagram of $\poset{L}(I)\setminus\hat 0$ 
is given below. 
Examining the filters 
$\bigl(\poset{L}(I)\setminus\hat 0\bigr)_{<x}$ for all 
$x\in\poset{L}(I)\setminus\hat 0$ shows 
that for any field $\gf$ the Betti poset 
$\poset{B}(I,\gf)$ consists of the non-circled elements. 
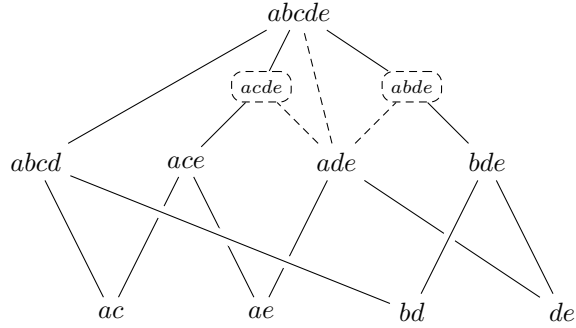
\[
\begin{tikzpicture}
[notinB/.style={rectangle, rounded corners=1.5mm, 
                densely dashed, draw=black}]
\node (abcde)          at (-.5,4) {$abcde$};
\node (acde)  [notinB] at  (-1,3) {$\scriptstyle acde$};
\node (abde)  [notinB] at  (+1,3) {$\scriptstyle abde$};
\node (abcd)           at  (-4,2) {$abcd$};
\node (ace)            at  (-2,2) {$ace$};
\node (ade)            at   (0,2) {$ade$};
\node (bde)            at  (+2,2) {$bde$};
\node (ac)             at  (-3,0) {$ac$};
\node (ae)             at  (-1,0) {$ae$};
\node (bd)             at  (+1,0) {$bd$};
\node (de)             at  (+3,0) {$de$};
\draw (de) to (ade) to (ae) to (ace) to (ac) to (abcd) 
           to (abcde) to (abde) to (bde) to (de)
      (abcde) to (acde) to (ace);
\draw[preaction={draw=white, -, line width=5pt}] 
      (bde) to (bd) to (abcd);
\draw[densely dashed] (ade) to (acde)
                      (ade) to (abde)
                      (ade) to (abcde);
\end{tikzpicture}
\]
When we investigate in a similar manner 
the monomial ideal $J=(wx,xy,wz,yz)$
in the polynomial ring 
$S=\gf[w,x,y,z]$, we see that 
for every field $\gf$ we have 
$\poset{B}(J,\gf)=\poset{L}(J)\setminus\hat 0$, 
with the following Hasse diagram.
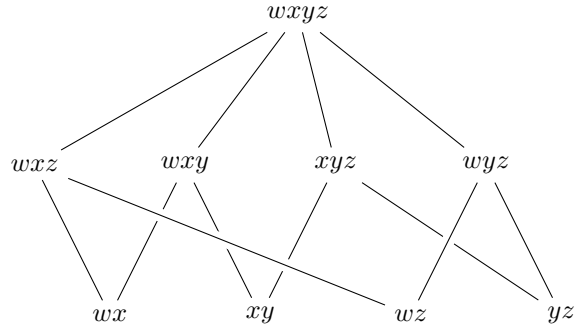
\[
\begin{tikzpicture}
\node (abcde)          at (-.5,4) {$wxyz$};
\node (abcd)           at  (-4,2) {$wxz$};
\node (ace)            at  (-2,2) {$wxy$};
\node (ade)            at   (0,2) {$xyz$};
\node (bde)            at  (+2,2) {$wyz$};
\node (ac)             at  (-3,0) {$wx$};
\node (ae)             at  (-1,0) {$xy$};
\node (bd)             at  (+1,0) {$wz$};
\node (de)             at  (+3,0) {$yz$};
\draw (de) to (ade) to (ae) to (ace) to (ac) 
           to (abcd) to (abcde) to (bde) to (de)
      (ace) to (abcde) to (ade);
\draw[preaction={draw=white, -, line width=5pt}] 
      (bde) to (bd) to (abcd);
\end{tikzpicture}
\]
Thus we have $\poset{L}(I)\not\cong\poset{L}(J)$ but 
$\poset{B}(I,\gf)\cong\poset{B}(J,\gf)$ for every field 
$\gf$. 
\end{example}

\begin{remark} 
In the rather simple example above, there exists a join preserving 
map from $\poset L(I)$ to $\poset L(J)$ that is a bijection 
on the atoms, and therefore 
it is still possible to obtain the minimal free resolution 
of $J$ by applying the relabeling Construction 3.2 from 
\cite{Gasharov-Peeva-Welker} to the minimal free resolution 
of $I$. For a more sophisticated example, where that kind of 
join-preserving map does not exist on the level of lcm-lattices, we refer 
the reader to \cite{Clark-Mapes2}*{Example~2.3}. 
\end{remark}


\section{When is a poset the Betti poset of an ideal?}
\label{S:when-is-betti}

Given a finite poset $\poset{P}$ and $x\in \poset{P}$, 
we write $A_x$ for the set of all minimal elements 
of $\poset{P}$ that are less than 
or equal to $x$. The poset  $\poset{P}$ is called 
\emph{atomic} if each $x\in\poset{P}$ is the join (the unique 
least upper bound) in $\poset{P}$ of the elements of $A_x$.   
We consider 
every atomic poset $\poset{P}$ with set of minimal elements 
$A$ naturally 
as a subposet of the Boolean lattice $\Sigma(A)$ 
(the set of all subsets of $A$, ordered by inclusion) 
via the embedding map $\sigma\colon \poset{P} \lra \Sigma(A)$ 
given by $\sigma(x)=A_x$; in particular we will not 
distinguish between an element $a\in A$ and the singleton 
$\{a\}\in\Sigma(A)$. It is an easy  
exercise for the reader to check that $\sigma$ 
preserves meets (unique greatest lower bounds, 
whenever they exist) of elements in $\poset{P}$. 
When $\poset{P}$ is atomic, we write 
$M(\poset{P})$ for the subposet  
of $\Sigma(A)$ with elements all meets in $\Sigma(A)$ 
of subsets of $\poset{P}$. 

\begin{lemma}\label{T:atomic-lattice}
Let $\poset{P}$ be a finite atomic poset with set of 
minimal elements $A$. Then $M(\poset{P})$ 
is a finite atomic lattice with set of atoms $A$. 
\end{lemma}

\begin{proof}
$M(\poset{P})$ is a finite meet-semilattice by construction, 
hence a finite lattice. 
Let $Z\in M(\poset{P})$. Since $Z\subseteq A$ and for 
each $a\in A$ the set $\{a\}$ is an atom of $M(\poset{P})$,  
it is enough to show that 
$Z=\bigvee_{a\in Z}\{a\}$. Let $V=\bigvee_{a\in Z}\{a\}$.  
Since joins in $M(\poset{P})$ are greater than or equal 
to joins in $\Sigma(A)$, we get that $V\ge Z$. 
On the other hand $Z$ is trivially an upper bound 
in $M(\poset{P})$ for the set 
$\bigl\{\{a\}\mid a\in Z \bigr\}$,  
hence $Z\ge V$.   
\end{proof}

Let $I$ be a monomial ideal in the polynomial ring  
$R=\gf[x_1,\dots, x_m]$ over the field $\gf$, and 
let $\poset{B}$ be its Betti poset over $\gf$. Thus, the 
set $A$ of minimal elements of $\poset{B}$ is exactly the 
set of degrees 
in $\IZ^m$ of the minimal generators of $I$. 
Let $\poset{L}=\poset{L}(I)$ be the lcm-lattice of $I$. 
In particular, the poset $\poset{L}\setminus\hat 0$ is 
atomic with 
set of minimal elements $A$, contains $\poset B$,  
and by 
\cite{Gasharov-Peeva-Welker}*{Theorem~2.1} 
an element 
$y\in\poset{L}\setminus\hat 0$ is not 
in $\poset{B}$ exactly when the relative homology 
$\RH_n\bigl(\Delta(\poset{L}_{<y}\setminus\hat 0),\gf\bigr)=0$ 
for all $n$.

\begin{lemma}
The Betti poset $\poset{B}$ is an atomic poset. 
\end{lemma}

\begin{proof} 
Let $x\in\poset{B}$, then clearly $x$ is an upper bound for 
the set $A_x$. Let $y\in\poset{B}$ be any other upper bound 
for $A_x$. Then $y$ is also an upper bound for $A_x$ in 
the lcm-lattice $\poset{L}$. Therefore $y$ is greater  than 
or equal to the join in $\poset{L}$ of the elements in $A_x$. 
Since $\poset{L}$ is join-generated by the elements of $A$ 
it follows that $z=\bigvee A_z$ for each 
$z\in\poset{L}$; in 
particular $\bigvee A_x=x$. Therefore $y\ge x$.  
\end{proof}

\begin{proposition}\label{T:if-Betti}
Let $I$ be a monomial ideal 
with Betti poset $\poset{B}$. 
Let $A$ be the set of 
minimal elements of $\poset{B}$ and let $M(\poset{B})$ be 
the corresponding subposet of $\Sigma(A)$. 
 
Then for each $x\in M(\poset{B})\setminus\hat 0$ we have that 
the element $x$ 
is not in $\poset{B}$ precisely when 
$
\RH_n\bigl(\Delta(M(\poset{B})_{<x}\setminus\hat 0);\gf\bigr)=0
$ 
for all $n$.
\end{proposition}

\begin{proof}
Let $\poset{L}$ be the lcm-lattice of $I$. 
Since $\poset{L}$ is an atomic lattice, the embedding of 
$\poset{L}$ inside $\Sigma(A)$ preserves meets and therefore  
$M(\poset{L})=\poset{L}$. It follows that 
$
\poset{B}\subseteq M(\poset{B})\setminus\hat 0
         \subseteq\poset{L}\setminus\hat 0\subset\deg(I).
$ 
Now the assertion of the proposition is immediate from 
Theorem~\ref{T:main-2} applied to 
$\poset{P}=M(\poset{B})\setminus\hat 0$. 
\end{proof}

\begin{theorem}\label{T:when_is_betti} 
Let $\poset{P}$ be a finite atomic poset with set of minimal 
elements $A$, let $M(\poset{P})$ be the subposet  
of $\Sigma(A)$ meet-generated by $\poset{P}$ in $\Sigma(A)$, 
and let $\gf$ be a field. 

Then 
$\poset{P}$ is the Betti poset of a monomial ideal  
over $\gf$ 
if and only if an element $x\in M(\poset{P})\setminus\hat 0$ 
is not in $\poset{P}$ precisely when 
$
\RH_n\bigl(\Delta(M(\poset{P})_{<x}\setminus\hat 0);\gf\bigr)=0
$ 
for all~$n$.
\end{theorem}

\begin{proof} 
The ``only if'' direction of the theorem is 
Proposition~\ref{T:if-Betti}. Suppose now for each 
$x\in M(\poset{P})\setminus\hat 0$ that 
$x\notin\poset{P}$ exactly when 
$
\RH_n\bigl(\Delta(M(\poset{P})_{<x}\setminus\hat 0);\gf\bigr)=0
$ 
for each $n$. Since $M(\poset{P})$ is an atomic lattice 
by Lemma~\ref{T:atomic-lattice}, 
it is the lcm-lattice of some monomial ideal $J$ \cite{Phan}. 
Thus by 
\cite{Gasharov-Peeva-Welker}*{Theorem~2.1} 
the poset $\poset{P}$ is the Betti poset of $J$ over $\gf$. 
\end{proof}

When $\poset B$ is the Betti poset of a monomial ideal $I$ over $\gf$ 
the lattice $M(\poset B)$ seems to play an important structural role. 
This motivates the following definition. 

\begin{definition} 
Let $\gf$ be a field, and let $I$ be a monomial ideal in $R$. We 
call the lattice $M\bigl(\poset B(I,\gf)\bigr)$ the 
\emph{Betti lattice} of $I$ over $\gf$. 
\end{definition}

As seen in the proof of Proposition~\ref{T:if-Betti},
one has $M(\poset B(I,\gf))\subseteq \poset L(I)$ for each $\gf$,   
yielding a collection of subposets of $\poset L(I)$. 
It is an intriguing open problem to investigate how the 
properties of this collection of Betti lattices 
reflect the properties of the ideal $I$.






\end{document}